\documentclass[12pt]{amsart}
\usepackage[colorlinks=true, pdfstartview=FitV, linkcolor=blue, citecolor=green, urlcolor=black,filecolor=magenta]{hyperref}

\newcommand{\aaa}{{\mathcal A}}

\newcommand{\vvv}{{\mathcal V}}

\newcommand{\T}{{\bf T}}
\newcommand{\TM}{\mathbb T}
\newcommand{\R}{{\mathbb R}}

\newcommand{\N}{{\mathbb N}}

\newcommand{\Z}{{\mathbb Z}}
\newcommand{\C}{{\mathbb C}}

\newcommand{\sP}{strongly pattern equivariant}
\newcommand{\wP}{weakly pattern equivariant}

\newtheorem{thm}{Theorem}[section]
\newtheorem*{thm*}{Theorem}
\newtheorem{cor}[thm]{Corollary}
\newtheorem{lem}[thm]{Lemma}
\newtheorem{prop}[thm]{Proposition}

\theoremstyle{definition}

\begin{document}
\title{Meyer sets, topological eigenvalues, and Cantor fiber bundles}
\author{Johannes Kellendonk and  Lorenzo Sadun}
\date{\today}
%\version{11}

\address{Johannes Kellendonk\\ Universit\'e de Lyon,
Universit\'e Claude Bernard Lyon 1\\
Institute Camille Jordan, CNRS UMR 5208\\  69622 Villeurbanne, France}\email{kellendonk@math.univ-lyon1.fr}
\address{Lorenzo Sadun\\Department of Mathematics\\The University of
 Texas at Austin\\ Austin, TX 78712 USA} 
\email{sadun@math.utexas.edu}
\thanks{
The work of the first author is partially supported by the ANR SubTile
NT09 564112. 
The work of the second author is partially supported by NSF
  grant  DMS-1101326} 
\date{June 7, 2013}
 % Math.Subj.Class. 

\begin{abstract}
  We introduce two new characterizations of Meyer sets.  A repetitive
  Delone set in $\R^d$ with finite local complexity is topologically
  conjugate to a Meyer set if and only if it has $d$ linearly
  independent topological eigenvalues, which is if and only if it is
  topologically conjugate to a bundle over a $d$-torus with totally
  disconnected compact fiber and expansive canonical
  action. ``Conjugate to'' is a non-trivial condition, as we show that
  there exist sets that are topologically conjugate to Meyer sets but
  are not themselves Meyer.  We also exhibit a diffractive set that is
  not Meyer, answering in the negative a question posed by Lagarias,
  and exhibit a Meyer set for which the measurable and topological
  eigenvalues are different.
\end{abstract}

\maketitle

\setlength{\baselineskip}{.6cm}

%\setcounter{tocdepth}{3}
%\makeatletter
%\def\l@subsection{\@tocline{2}{0pt}{25pt}{5pc}{}}
%\def\l@subsubsection{\@tocline{2}{0pt}{50pt}{5pc}{}}
%\makeatother
%\tableofcontents

\section{Introduction}
To provide a rigorous mathematical explanation of the observation that
certain non periodic media (quasicrystals) show sharp Bragg peaks in
their X-ray diffraction, mathematicians came up with the notion of a
{\em pure point diffractive} set. This is a point set $\Lambda$ of
$\R^d$ to which can be associated an auto-correlation measure $\gamma$
whose Fourier transform $\hat\gamma$ is a pure point measure \cite{Hof} 
(see also \cite{Moody08} and references therein). In the
model of the material by the point set, the Bragg peaks correspond
exactly to the points in the dual space ${\R^d}^*$ that have strictly positive
$\hat\gamma$-measure.  This raises the question of which point sets are
pure point diffractive. Whereas an answer in terms of the properties
of the autocorrelation measure $\gamma$ has been found, namely that
this is the case whenever $\gamma$ is strongly almost periodic
\cite{BM,Gouere,GA} no complete geometric characterization of 
such point sets is known.
 
The question %of whether a point set is pure point diffractive 
can be reformulated as a property of the measurable dynamical system
$(\Omega,\R^d,\mu)$. Here $\Omega$ is the hull of $\Lambda$, a compact
metrizable space consisting of points sets whose patches look like
those of $\Lambda$ and on which $\R^d$ acts by translation, and $\mu$
is an invariant ergodic Borel probability measure that is determined
by the patch frequencies. A point $\beta\in{\R^d}^*$ is called an eigenvalue  (or dynamical eigenvalue) for $\Lambda$ if there exists an $L^2$-function $f$ (an eigenfunction)
that satisfies the eigenvalue equation
\begin{equation}\label{eq-EV}
f(\omega - t ) = e^{2\pi i\beta(t)} f(\omega)
\end{equation}
for all $t\in\R^d$ and $\mu$-almost all $\omega\in\Omega$.\footnote{
The dual space ${\R^d}^*$ is often identified with $\R^d$ using a scalar product. Then $\beta(t)$ 
is the scalar product of $\beta$ with $t$.} 
Arguments based on work of Dworkin \cite{Dworkin,LMS,BL}  
showed that any Bragg peak gives rise to a dynamical eigenvalue. Moreover,
$\Lambda$ is pure point diffractive whenever the measurable dynamical system $(\Omega,\R^d,\mu)$ has pure point dynamical spectrum, that is, whenever
$L^2(\Omega,\mu)$ is spanned by eigenfunctions. 

Interestingly enough, there are classes of point sets for which pure
point diffractivity becomes a property of {\em topological} dynamical
systems. An eigenvalue $\beta$ is called a continuous or {\em
  topological} eigenvalue if (\ref{eq-EV}) has a {\em continuous}
solution.  When the point set $\Lambda$ comes from a 
substitution it is known that that all eigenvalues are
topological \cite{Host, Solomyak}.
Furthermore, repetitive regular model sets, which are known to be
pure point diffractive, also have only topological eigenvalues  
\cite{Hof,Schlottmann}. Knowing that an eigenvalue is topological is of advantage in case that the measure $\mu$ is uniquely ergodic, as it allows us to compute the intensity of the associated Bragg peak by means of a Bombieri-Taylor formula \cite{Lenz}.

However, there do exist diffractive 
point sets whose measurable eigenvalues are not all topological 
(see Section \ref{Counter-ex}). It then makes sense to ask about the 
topological eigenvalues. For which point sets are all the eigenvalues 
continuous? 

An experimental material scientist may doubt the usefulness of the
mathematical concept of pure point diffraction for $X$-ray analysis,
since experimental devices only have finite resolution. Since we can only
see Bragg peaks above a given brightness $s$, we can never observe a dense set
of peaks. At best, we can observe a {\em relatively} dense set of peaks,
meaning there is a radius $r$ such that every ball of radius $r$ in $\R^{d*}$
contains at least one peak. 

Recently, Nicolae Strungaru \cite{St1,St} showed that Meyer sets have
this property for any $s$ below the maximal intensity of the Bragg peaks. In
other words, if you can see at least one Bragg peak from a Meyer set
$\Lambda$, and if you increase the sensitivity of your equipment even
slightly, then you will see a relatively dense set of Bragg peaks.  
From this point of view all Meyer sets should be regarded as being 
diffractive, but not in the sense of having {\em pure} point diffraction
spectrum. Instead, a point set with relatively dense set of Bragg peaks
is called {\em essentially diffractive} \cite{St}.

Dworkin's argument now yields a necessary criterion of essential
diffractivity: $\Lambda$ must have $d$ linearly
independent (measurable) eigenvalues. We then ask: is this criterion a
topological property for certain classes of point sets? In other
words, which point sets have $d$ linearly independent {\em
  topological} eigenvalues?  One of our main results gives a
characterization of such sets:
\begin{thm}\label{main}
  A repetitive Delone set with finite local complexity
has $d$ linearly independent topological
  eigenvalues if and only if it is topologically conjugate to a Meyer
  set. In particular, each repetitive Meyer set has $d$ linearly independent
  topological eigenvalues.
\end{thm}
Our theorem can be combined with a characterization of regular model sets given in \cite{BLM05} 
to obtain:
\begin{thm}\label{cor-main}
  A repetitive Delone set with finite local complexity is topologically conjugate to a repetitive 
  regular model set if and only if its dynamical system is uniquely ergodic and the continuous eigenfunctions separate almost all points of its hull.
%  the factor map from its hull onto its maximal equicontinuous factor is almost everywhere 1:1.
\end{thm}

Our analysis draws on the use of pattern equivariant functions and of
certain fiber bundles. 
Theorem~\ref{main} is actually proven via a
third characterization of repetitive Meyer sets, which may well have its own
benefits:
\begin{thm}\label{main2}
  The dynamical system of a repetitive Delone set of finite local complexity is
  topologically conjugate to a
  bundle over the $d$-torus with totally disconnected compact fiber and 
  expansive canonical $\R^d$-action if
  and only if it is topologically conjugate to the dynamical system of a Meyer set. 
\end{thm}

If we add the requirement that the Delone set is not fully periodic then the fiber
of the bundle is a Cantor set.

The above theorems characterize repetitive Delone sets up to topological
conjugacy, by which we mean topological conjugacy of their associated
dynamical systems. To obtain a geometric characterization one needs to
understand to what extent the dynamical system of a Delone set
determines the Delone set and which geometric propeties of a Delone
set are preserved under topological conjugacy. For this the following
theorem is useful.
\begin{thm}\label{main3}
  Any topological conjugacy between the dynamical systems of Delone
  sets of finite local complexity is the composition of a shape
  conjugation with a mutual local derivation.  Furthermore, the shape conjugation
  can be chosen to move points by an arbitrarily small amount.
\end{thm}
%Here a shape conjugacy is a shape deformation that is also a topological conjugacy. 

A natural question is whether a Delone set of
finite local complexity (FLC) that
is shape conjugate to a Meyer set must itself be a Meyer set. 
As a consequence of the above theorem, any FLC Delone set that is
shape conjugate to a Meyer set is arbitrarily close in the Hausdorff
distance to a Meyer set. However, this does not suffice to guarantee that it
is Meyer. Indeed, we provide a counterexample. The Meyer property is invariant
under local derivations but not under shape conjugacy.

It is not surprising that Meyer sets show up as solutions to the
topological version of essential diffractivity. They have been
investigated throughout the efforts to describe sets which are pure
point diffractive. Lagarias (see \cite{Lagarias}, Problem 4.10) 
suggested the problem of proving that a repetitive pure point diffractive set is 
necessarily Meyer. This turns out to be false. A recent example of \cite{FS}, 
called the scrambled Fibonacci tiling, provides a counterexample. 
It has pure point dynamical spectrum, but none of the eigenfunctions
can be chosen continuous, and hence it is not a Meyer tiling
(something that can also be checked directly). 

A variation of the scrambled Fibonacci yields a Meyer set that has dense pure
point dynamical spectrum but only relatively dense topological
dynamical spectrum. The dynamical spectrum has rank 2 while the topological
spectrum only has rank 1. 

\section{Preliminaries}

We assume that the reader is familiar with the standard objects and
concepts from the theory of tilings or Delone sets, in particular with
the notions of finite local complexity (FLC) and repetitivity. We also
suppose that the construction of the hull $\Omega_\Lambda$ and the
dynamical system $(\Omega_\Lambda,\R^d)$
associated with a tiling or a Delone set $\Lambda\subset\R^d$ is known.  
Where necessary we
assume the existence of frequencies of patches defined by a van Hove
sequence and hence of an ergodic probability measure $\mu$ on the
hull, see, for instance, \cite{Solomyak}.

An important well-known concept for us is that of a local derivation
(or local map) between two point sets or tilings. A 
point set $\Lambda'$ is locally derivable from a point set $\Lambda$ 
if there exists $R>0$ such that the question of whether a point $x$
belongs to $\Lambda'$ can be answered (without direct reference to $x$) 
by inspection of $B_R(0)\cap(\Lambda-x)$ \cite{BJS}. 
This defines a factor map between the hulls $\Omega_\Lambda$ and $\Omega_{\Lambda'}$
which is called a local derivation.
The definition applies literally to tilings if we
identify the tiling with a subset of $\R^d$ of boundary points of its
tiles. There are many ways to convert Delone sets to tilings or to
convert tilings to Delone sets in a mutually locally derivable way,
i.e.\ with bijective maps with are in both directions local derivations. We simply
say then that the two objects are MLD.  Given a (polyhedral) FLC
tiling $\T$, the set of its vertices is an FLC Delone set which is locally derivable from the tiling. 
Given an FLC Delone set $\Lambda$, the collection of Voronoi cells is an FLC
tiling, and the vertices of the dual to this Voronoi tiling (the so-called
Delone tiling) are precisely the elements of $\Lambda$. We denote by
$\T_\Lambda$ the Delone tiling of $\Lambda$ which is clearly mutually locally
derivable with $\Lambda$. 
 This paper is
written primarily in the language of Delone sets, but we freely use
theorems about tilings.

When we speak about eigenfunctions and eigenvalues 
for a Delone set $\Lambda$ (or a
tiling), what we mean are eigenfunctions for the system
$(\Omega_\Lambda,\R^d,\mu)$, and when we say that two Delone sets $\Lambda$
and $\Lambda'$ are topologically conjugate we mean that their
dynamical systems $(\Omega_\Lambda,\R^d)$ and $(\Omega_{\Lambda'},\R^d)$
are topologically conjugate.

\section{Meyer sets}
A subset $\Lambda\subset \R^d$ is called {\em harmonious} if any
algebraic character on the group $\langle \Lambda\rangle$ generated by
$\Lambda$ can be arbitrarily well approximated by a continuous
character. Continuity refers here to the relative topology of
$\langle\Lambda\rangle\subset \R^d$ and so a continuous character is
of the form $e^{2\pi i\beta}$ for some $\beta\in {\R^d}^*$. Very
roughly speaking, harmonious sets are (potentially) non-periodic sets to
which harmonic analysis can still be applied. Harmonious  sets
play an important role in Meyer's book
\cite{Meyer72}. In particular the class of harmonious sets which Meyer
called {\em model sets} (see \cite{Moody}) plays an important role in spectral synthesis
\cite{Meyer72} and in optimal and universal sampling in information
theory \cite{Meyer72,MateiMeyer08,Meyer12}. After the discovery of
quasicrystals, the relevance of harmonious sets and model sets to the
diffraction was recognized
\cite{Meyer95,Moody,Lagarias96} and nowadays a relatively dense
harmonious set is called a {\em Meyer set}.  A tiling is a Meyer
tiling if it is MLD to a Meyer set.

Harmonious sets allow for various different characterizations.
Some of them are analytical but there is also a surprisingly simple 
geometric criterion. 
Let $$\Lambda^\epsilon = \{ \beta \in \R^{d*} | \forall a \in \Lambda, 
|1 - \exp(2 \pi i \beta(a))| < \epsilon\}.$$ The following conditions
are equivalent \cite{Meyer72,Moody,Lagarias96,Lagarias99}.
\begin{enumerate}
\item{$\Lambda$ is harmonious, that is, for any group homomorphism $\phi:\langle\Lambda\rangle\to S^1$ and any $\epsilon>0$, there is a $\beta
\in \R^{d*}$ such that $|\phi(a) - \exp(2 \pi i \beta(a))|<\epsilon$ for
all $a \in \Lambda$.}
\item{$\Lambda^\epsilon$ is relatively dense for all $\epsilon>0$.}
\item{$\Lambda^\epsilon$ is relatively dense for some $0<\epsilon <1/2$.}\label{item-epsilon}
\item{$\Lambda - \Lambda$ is uniformly discrete.}\label{item-discrete}
\item{$\Lambda$ is uniformly discrete and there exists a finite set 
$F$ such that $\Lambda-\Lambda = \Lambda + F$.}
\item{$\Lambda$ is a subset of a model set.}
\end{enumerate}
The two characterizations we will use here are (\ref{item-epsilon}) and
(\ref{item-discrete}).  It follows from (\ref{item-discrete}) that a Meyer
set is an FLC Delone set. Furthermore, inside the class of Delone sets (\ref{item-discrete}) is preserved under local derivation \cite{Lagarias99} and so a Delone set which is MLD to a Meyer set is also a Meyer set.
Note that $\Lambda - \Lambda$ is always locally finite
if $\Lambda$ is FLC, but the condition that
$\Lambda - \Lambda$ is {\em uniformly} discrete is very strong: it
implies, for instance, that $\Delta=\Lambda-\Lambda$ is also a Meyer
set \cite{Moody}. Hence $\Delta^\epsilon$ is relatively dense
for all $\epsilon > 0$; that is, for a relatively dense set of
$\beta$'s we can then find a plane wave $e^{2\pi i\beta}$ that has
nearly the same phase at all points in $\Lambda$. This will
be the key to constructing continuous eigenfunctions.
%
%It is important to note that inside the class of Delone sets the Meyer property is preserved under MLD transformations \cite{Lagarias99}. Indeed, suppose that $\Lambda$ is a Meyer set and that $\Lambda'$ is a uniformly discrete set which is locally derivable from $\Lambda$. Fix $r>0$ and let $M_x:= B_r(0)\cap(\Lambda'-x)$ for $x\in\Lambda$. $\Lambda$ has finite local complexity. Thus $\Lambda'$ 
%being locally derivable from $\Lambda$ implies that when letting $x$ vary over $\Lambda$ we only get finitely many different sets $M_x$. In particular also $M=\bigcup_{x\in\Lambda}M_x$ is finite.
%It follows that $\Lambda-\Lambda + (M-M)$ is uniformely discrete. Clearly $\Lambda'-\Lambda' \subset \Lambda-\Lambda + (M-M)$, provided $r$ is large enough so that any $r$-ball contains a point of $\Lambda$.
\bigskip

Model sets are point sets obtained from a cut \& project scheme \cite{Moody}
with very strong properties which makes them suitable for the description of quasicrystals, in particular if they are regular\footnote{``Regular'' means that the boundary of the so-called window (or acceptance domain) has measure zero.} \cite{Schlottmann}. 
To our knowledge, no direct geometric characterization of model sets is known. Using their associated dynamical systems, regular model sets can be characterized in the following way:
\begin{thm}[\cite{BLM05}]
$(\Omega,\R^d)$ is the dynamical system of a repetitive regular model set if and only if
\begin{enumerate}
\item all elements of $\Omega$ are Meyer sets,
\item $(\Omega,\R^d)$ is minimal and uniquely ergodic,
\item $(\Omega,\R^d)$ has pure point dynamical spectrum and all eigenvalues are topological, and
\item the continuous eigenfunctions separate almost all points of $\Omega$.
\end{enumerate}  
\end{thm}
Since the maximal equicontinuous factor of a dynamical system is the spectrum of the algebra generated by its continuous eigenfunctions, the last property is equivalent to the fact that
the factor map $\pi_{max}:\Omega\to \Omega_{max}$
from $\Omega$ to its maximal equicontinuous factor $\Omega_{max}$ is almost everywhere 1:1, 
and hence induces a measure isomorphism between $L^2(\Omega,\mu)$ and $L^2(\Omega_{max},\eta)$ ($\eta$ is the Haar measure on $\Omega_{max}$) (see e.g.\ \cite{BargeKellendonk}). Taking into account that $C(\Omega_{max})$ is dense in  $L^2(\Omega_{max},\eta)$, it follows that 
Property~4 implies Property~3.
Moreover, by Theorem~\ref{main}, Property~3 implies that 
$(\Omega,\R^d)$ is topologically conjugate to the dynamical system of a Meyer set; hence
the above theorem simplifies to Theorem~\ref{cor-main}.

\section{Pattern equivariant functions and cochains}
We recall the definitions of pattern equivariant functions and cochains for
uniformly discrete FLC point sets or FLC tilings \cite{K,integer}. 

Let $\Lambda$ be a uniformly
discrete point set or tiling in $\R^d$ of finite local complexity. We call a
function $f:\R^d\to \C$ (or into any abelian group $A$) \sP\
if there exists a radius $R$ such that $f(x)=f(y)$ 
for any two points 
$x,y$ with $(\Lambda-x)\cap B_R(0) = (\Lambda-y)\cap B_R(0)$.  
In other words, for each $x$, 
$f(x)$ is determined by the pattern of points within a 
fixed finite distance $R$ of $x$. 
We can also consider a tiling 
as a decomposition of $\R^d$ into 0-cells,
1-cells, $\ldots,$ and $d$-cells. A $k$-cochain assigns a number to each
$k$-cell. Such a cochain is called \sP\ if the 
value of a $k$-cell depends only on the pattern of tiles within a distance
$R$ of that $k$-cell. 

An equivalent definition of pattern equivariance involves the 
description of the hull $\Omega_\Lambda$ as the inverse limit of
approximants $\Gamma_R$ 
that describe the point set $\Lambda$ 
(or tiling) out to distance $R$ around the
origin. There is a natural map from $\R^d$ to $\Gamma_R$ sending $x$ to the
equivalence class of $\Lambda-x$. A function or cochain is \sP\
if it is the pullback of a function or cochain on a fixed approximant
$\Gamma_R$.

Weakly pattern equivariant functions and cochains are defined as 
limits of \sP\ objects. The precise definition
depends on the category that we are working in.
\begin{itemize}
\item A continuous function $f$ is \wP\  (in the topological sense)
if it can be approximated  in the uniform topology
  by strongly pattern equivariant continuous functions, i.e.\ for all
  $\epsilon>0$ there exists a strongly pattern equivariant continuous
  function $h$ such that $\|f-h\|_\infty < \epsilon$.
\item A cochain $f$ on the tiling $\T$ derived from $\Lambda$ is \wP\
  if it can be approximated in the sup-norm by
  \sP\ cochains, i.e.\ for all $\epsilon>0$
  there exists a strongly pattern equivariant cochain $h$ such that
  $\sup_{c}|f(c)-h(c)| < \epsilon$.
  \end{itemize}

\subsection{Dynamical eigenfunctions}
Pattern equivariant functions can be used to describe
eigenfunctions, i.e.\ solutions to equations (\ref{eq-EV}).  
%Given a function $f:\Omega\to \C$ and $\omega\in\Omega$, we define
%$f_\omega:\R^d\to \C$ by $f_\omega(t) := f(\omega-t)$. 
\begin{lem} \label{prop-EV}
Let $\Lambda$ be a repetitive FLC Delone set. $\beta$ is a topological eigenvalue for $\Lambda$ if and only if $e^{2\pi i \beta}$ is \wP\ in the topological sense.
\end{lem}
\begin{proof}
It has been known for some time that \wP\ functions in the topological sense
are precisely the restrictions of continuous functions on $\Omega_\Lambda$ to the
orbits through $\Lambda$.\footnote{This statement is e.g.\ implicit in \cite{K}
in the proof that \wP\ functions in the topological sense are isomorphic, as a $C^*$-algebra, to
$C(\Omega_\Lambda)$.} Clearly, any continuous solution of  (\ref{eq-EV}) restricts to a function proportional to $e^{2\pi i \beta}$ on the orbit through $\Lambda$.
\end{proof}

The following result is useful for determining whether 
an element $\beta\in{\R^d}^*$ is a topological
eigenvalue.
\begin{lem}\label{lem-wP}
Let $\Lambda\in\R^d$ be an FLC Delone set and let $f:\R^d\to \C$ be a continuous function.
$f$ is \wP\ if and only if it is uniformly continuous and for all $\epsilon>0$ there exists $R>0$ such that
$(\Lambda-x)\cap B_R(0) = (\Lambda-y)\cap B_R(0)$ implies $|f(x)-f(y)|<\epsilon$.
\end{lem}
\begin{proof}
A \wP\ continuous function extends to a continuous function on the hull of $\Lambda$. Since the hull is compact, this extension is 
uniformly continuous, implying that the original function is also uniformly continuous. To see the second property, 
approximate $f$ to within $\frac\epsilon 2$ by a \sP\ continuous function.

For the converse, let $f$ be a uniformly continuous function satisfying the second property of the lemma.  
Pick $\epsilon>0$ and $R>0$ such that
$(\Lambda-x)\cap B_R(0) = (\Lambda-y)\cap B_R(0)$ implies $|f(x)-f(y)|<\epsilon$. Without loss of generality we may assume that $f$ is real valued.
Define
$$f_\epsilon(x) := \inf\{f(y):(\Lambda-x)\cap B_R(0) = (\Lambda-y)\cap B_R(0)\}.$$
Then $f_\epsilon$ is \sP\ and $|f_\epsilon(x) - f(x)|\leq \epsilon$. Moreover, the uniform continuity of $f$ implies that $f_\epsilon$ is continuous. 
\end{proof}

The next two results concern functions $F$
that are {\em a priori} only defined on $\Lambda$ and therefore may be regarded as 
$0$-cochains on the tiling $\T_\Lambda$. The coboundary $\delta F$ of
$F$ is then a $1$-cochain on $\T_\Lambda$ that computes the change in $F$ along an arbitrary chain.

\begin{prop} \label{prop-2}
Suppose that $\Lambda$ is FLC and repetitive. 
Let $F:\Lambda\to\R$ be a bounded function such that the 
$1$-cochain $\delta F$ on $\T_\Lambda$ is \sP. 
Then for all $\epsilon$ there exists a relatively dense set $\Lambda_\epsilon$
which is locally derivable from $\Lambda$ such that
$\forall p,q\in \Lambda_\epsilon$: $|F(p)-F(q)|<\epsilon$.
\end{prop}
\begin{proof} Let $2M= \sup_{a,b\in\Lambda} (F(b)-F(a))$. Then $F$ takes values in
  $[f_0-M,f_0+M]$ for some $f_0$. Let $f = F-f_0$.  Given $\epsilon>0$ there are
  $a,b$ such that $2M-\epsilon\leq f(b)-f(a)\leq 2M$. Hence
  $M-\epsilon\leq f(b)\leq M$ and $-M\leq f(a)\leq -M+\epsilon$. Let
  $c$ be a chain with boundary $b-a$. Since $\delta F$ is strongly
  pattern equivariant, $c$ is contained in a patch $\tilde c$ 
  such that the value $\delta F(c)$ depends only on $\tilde c$; in
  other words, whenever $\tilde c + t$ occurs in the complex then
  $\delta F(c+t) = \delta F(c)$. Let $P$ be the set of $t$ such that
  $\tilde c + t$ occurs in the complex.  $P$ is locally derived from $\Lambda$,
  and by repetitivity $P$ is
  relatively dense. Taking $\Lambda_\epsilon=P+b$ and $p\in\Lambda_\epsilon$, 
  %It follows that the set $\Lambda_\epsilon$ of
  %vertices such that for all $b\in \Lambda_\epsilon$ we have
  we have $f(p)=f(p-b+a)+ \delta F(c) \ge -M + (2M-\epsilon) = M-\epsilon$, 
%$M-\epsilon\leq f(p)\leq M$ for all $p \in \Lambda_\epsilon$, and
  hence $|f(p)-f(q)|<\epsilon$ for all $p,q \in \Lambda_\epsilon$. $\Lambda_\epsilon$
  is locally derived from $\Lambda$ and relatively dense.  
  \end{proof}
%Prop.~\ref{prop-2} together with a discrete variant of Lemma~\ref{lem-wP} yield a criterion for weak exactness of \sP\ 1-cochains which is of independent interest.
\begin{cor} \label{cor-exact} 
Under the conditions of Prop.~\ref{prop-2} $F$ is \wP.
\end{cor}
Note that boundedness of $F$ is also a necessary criterion since any \wP\ cochain is bounded. 
\begin{proof}
  Applying Prop.~\ref{prop-2} to $F$ we find that given any
  $\epsilon$, there exists a relatively dense set $\Lambda_\epsilon$
  that  is locally derivable from $\Lambda$ such that $\forall p,q\in
  \Lambda_\epsilon$: $|F(p)-F(q)|<\epsilon$.  We partition $\Lambda$ into
  subsets $\{V_a\}_a$ for $a\in\Lambda_\epsilon$ as follows. Let
  $\vvv(a)$ be the Voronoi domain of $a\in\Lambda_\epsilon$. In a
  generic situation, no point of $\Lambda$ lies on the boundary of
  some $\vvv(a)$ and then we can define $V_a = \vvv(a)\cap\Lambda$. In
  a non-generic situation we can make choices to associate the points
  on a boundary to one of the Voronoi domains in a locally derivable
  way (e.g.\ by a directional criterion). As a result, we have
  obtained a partition of $\Lambda$ that  is locally derivable from
  $\Lambda$. Now define
$$F_\epsilon(a) := \inf\{F(b):b\in\Lambda_\epsilon\}$$
for $a\in \Lambda_\epsilon$ and 
$$ F_\epsilon(x) = F_\epsilon(a) + \delta F(c) $$
for $x\in V_a$
where $c$ is any 1-chain with boundary $x-a$. Since $\delta F$ is \sP\ and the Voronoi domains are bounded we obtain that $F_\epsilon$ is \sP. Furthermore,
also $ F(x) = F(a) + \delta F(c) $ so that $|F(x) - F_\epsilon(x)|\leq \epsilon$.
\end{proof}

\section{Topological conjugacies}
\subsection{Shape deformations and shape conjugacies} 
 Consider two FLC polyhedral tilings $\T$ and $\T'$.  Each tiling is
 determined by its edges and by the location of a single vertex. Deforming the edges deforms the shapes
 and sizes of the tiles.  $\T'$ is a (shape) deformation of $\T$ if
 $\T$ and $\T'$ have the same combinatorics, and if the vectors that
 give the displacements along the edges of $\T'$ are obtained in a
 local way from the corresponding edges of $\T$. That is, there exists
 $R>0$ such that the displacement along edge $e'$ of $\T'$,
 corresponding to edge $e$ of $\T$, depends only on $B_R(e)\cap
 \T$. This does not mean that inspection of $B_R(e)\cap \T$ allows one
 to determine the location of $e'$, only the relative position of the
 two endpoints. 
 We can encode this by a vector-valued $1$-form $\alpha$ on $\T$, namely,
 $\alpha(e) = v_{e'} $ where $v_{e'}$ is the displacement vector along edge $e'$. 
 The requirement that $v_{e'}$ depends only on $B_R(e)\cap \T$ means that $\alpha$ is \sP.
Note that shape deformations automatically preserve finite local complexity.

There are two canonical ways to turn a shape deformation between two specific tilings $\T$, $\T'$
into a continuous surjection $\Omega_\T\to\Omega_{\T'}$.
The first method preserves canonical transversality and applies to all shape deformations \cite{SW,K2}. For simplicity we suppose that $\T$ has a vertex $x_0$ on the origin $0\in\R^d$.
Consider the vector-valued $0$-cochain $\tilde F$ satisfying $\delta \tilde F = \alpha$ and $\tilde F(x_0)=0$.
Let $h: \R^d \to  \R^d$ be a piecewise-linear 
extension\footnote{Identifying the vertices of $\T$ with a subset of $\R^d$ we view $\tilde F$ as a map on this subset.}
of $\tilde F$ to $\R^d$, and let $\phi_1(\T-x) = \T' - h(x)$.  
This map is uniformly continuous on the orbit of $\T$, and so extends to a continuous surjection 
$\phi_1:\Omega_\T\to\Omega_{\T'}$. Note that $\phi_1$ does not commute with translations, precisely because the distances between vertices in $\T'$ are different from the distances between vertices in $\T$. 

The second method only applies to very special shape deformations (namely those that are called asymptotically negligible in \cite{CS2,K2}).
Let $F=\tilde F - \mbox{id}$. $F$ is the vector-valued 0-cochain on $\T$ whose value at an arbitrary vertex $x \in \T$ is $x'-x$, where $x'$ is the corresponding vertex of $\T'$. 
If $F$ is \wP, then $F$ extends to a continuous map on $\Omega_\T$ implying that
$\phi_2(\T-x) = \T'-x$ is also uniformly continuous. $\phi_2$ thus extends to a continuous surjection
$\phi_2: \Omega_\T \to\Omega_{\T'}$. In fact, if $\T_1 \in \Omega_\T$, 
then $\phi_2(\T_1)$ is the tiling obtained 
by moving each vertex $x$ of $\T_1$ by $F(x)$. By construction, this map commutes with translations and hence is a topological semi-conjugacy.
We call this a {\em shape semi-conjugacy}, and if $\phi_2$ is invertible a {\em shape conjugacy}.  If $F$ is \sP, then the shape semi-conjugacy is a local derivation. Local derivations preserve the
Meyer property, but we will see that general shape deformations do not. 

The $0$-cochain $F$ is determined by the tilings $\T$ and $\T'$, but we can equally well construct shape semi-conjugacies directly from cochains. 
Let $F$ be any \wP\ vector-valued 0-cochain taking values that are less than half the separation of vertices in $\T$, such that $\delta F$ is \sP. Then we can construct $\T'$ from $\T$ by moving each vertex $x$ of $\T$ by $F(x)$ and preserving the combinatorics. It follows from the above that
there is a semi-conjugacy $\phi_2$ from $\Omega_\T$ to $\Omega_{\T'}$. If $F$ is small enough then this procedure can be inverted \cite{K2} so that $\phi_2$ is a shape conjugacy.

The following theorem implies Theorem~\ref{main3}, showing that shape conjugacies are essentially the
only topological conjugacies that are not mutual local derivations. 

\begin{thm}\label{thm-shape}
 Let $\Lambda$ and $\Lambda'$ be FLC Delone sets which are pointed topologically conjugate, i.e.\ there exists a topological conjugacy $\phi: \Omega_\Lambda \to \Omega_{\Lambda'}$ with $\phi(\Lambda) = \Lambda'$.  For each
  $\epsilon >0$ there exists a mutual local derivation $\psi_\epsilon$ and a
  shape conjugacy $s_\epsilon$ such that $\phi = s_\epsilon \circ
  \psi_\epsilon$  and $s_\epsilon$ moves the location of
  each point in each pattern by less than $\epsilon$.
\end{thm}

\begin{proof} 
%Pick a reference point pattern $\Lambda_0 \in \Omega$ and let $\Lambda' = \phi(\Lambda_0)$.
Let $R'$ be the minimum distance between distinct points of $\Lambda'$, and
pick $\epsilon < R'/3$. Since $\phi$ is uniformly continuous, there exists
a $\delta$ such that any two tilings within distance $\delta$  
map to tilings within $\epsilon$ of one another. Here we use the common distance between patterns: $\Lambda_1$ and $\Lambda_2$ have distance at most 
$\delta$ if their patterns in a ball of radius $1/\delta$ agree up to a translation of size at most $\delta$. Thus by uniform continuity
of $\phi$ there exists an $R = 1/\delta$ such that the locations of the points in 
  $\Lambda'$ in a ball of radius $1/\epsilon$ around $x\in\R^d$
are determined to within $\epsilon$ by the pattern of
%  $\Lambda_0$ 
$\Lambda$  in a ball of radius $R$ around those points.
%{\tt this sentence needs more explanation--J.} 
In particular, 
each point in $\Lambda'$ (say at location $x_0$) 
is associated with a pattern of radius $R$ around $x_0$ in
$\Lambda$. 
Translate these patterns by $-x_0$ to yield a set $S$ of point patterns 
on the ball of radius $R$ around the origin. Note that if we have 
two such point patterns $P_{1,2}$ with $P_1=P_2 - y$, 
then either $|y|<\epsilon$ or $|y|>R'-\epsilon$. (The latter can
happen if $P_1$ determines the existence of several points of $\Lambda'$.) 
Since $R'-\epsilon$ is strictly greater than $2\epsilon$, there is 
an equivalence relation on $S$ that two patterns are equivalent if they
are translates by less than $\epsilon$.

Since $\Lambda$ is FLC, there are only finitely many equivalence classes in $S$. 
Pick representatives $\{P_1, \ldots, P_n\}$ of the equivalence classes. 
Define a point pattern $\Lambda_\epsilon$ in the following way:
$\Lambda_\epsilon$ has a point at $x$ if and only if one of the patterns $P_i+x$ appears
in $\Lambda$. By construction $\Lambda_\epsilon$ is locally derivable from $\Lambda$.
Moreover, the points of $\Lambda_\epsilon$ are in 1-1 correspondence
with the points of $\Lambda'$, and each point in $\Lambda_\epsilon$
is within $\epsilon$ of the corresponding point of $\Lambda'$.
Let $\psi_\epsilon:\Omega_\Lambda\to \Omega_{\Lambda_\epsilon}$ be the local derivation
defined by the above procedure.

We show that $\Lambda_\epsilon$ is actually MLD with $\Lambda$ and $\psi_\epsilon$ is a mutual local derivation. In fact, 
since $\phi$ is a topological conjugacy, the pattern of $\Lambda$ on
a ball of radius $R$ is determined, up to translation by $\epsilon$, by 
the pattern of $\Lambda'$ on a bigger ball, say of radius $R''$. 
This means that the pattern $P_i+x$ in $\Lambda$ that generated a 
point $x \in \Lambda_\epsilon$ can be determined from a ball of radius
$R''+\epsilon$ around $x$ in $\Lambda'$. Thus $\psi_\epsilon$ is injective.
Since $\psi_\epsilon$ is an injective 
local derivation, $\psi_\epsilon$ is a mutual local derivation. 

We claim that the map $s_\epsilon = \phi \circ \psi_\epsilon^{-1}$ is
a shape conjucagy. 
%It is clearly a conjugacy and so we only have to
%show that it is induced from a strongly pattern equivariant $1$-cochain.
Let $F$ be the 0-cochain
on $\T_{\Lambda_\epsilon}$ whose value on each
point $x\in \Lambda_\epsilon$ is the displacement to the corresponding point in $\Lambda'$. 
By the above, $\Lambda_\epsilon$ and $\Lambda$ are at most distance $\epsilon$ apart. So if $\epsilon$ is small enough then $\T_{\Lambda_\epsilon}$ and $\T_{\Lambda'}$ have the same combinatorics.\footnote{Except in the situation that $\Lambda'$ has accidentally high symmetries, causing the Voronoi cells of $\Lambda'$ to meet in non-generic ways. In such cases a 
small perturbation of $\Lambda'$ can change the combinatorics of the Voronoi tiling, and hence of the Delone tiling. Such cases can be dealt with by working with a (subsequential) 
limit tiling $\lim_{\epsilon\to 0} \T_{\Lambda_\epsilon}$, whose vertices are still $\Lambda'$,  instead of with $\T_{\Lambda'}$.} 
We show that the coboundary $\delta F$ is  \sP: Since $\Lambda_\epsilon$ and $\Lambda$ are MLD, a pattern in a quite large ball in  $\Lambda_\epsilon$
determines the pattern in a large ball in $\Lambda$. We also saw that uniform continuity of $\phi$
implies that a  large ball in $\Lambda$ determines the pattern in $\Lambda'$ up to a small
overall translation. But overall translations don't matter for coboundaries and so the very large ball in  $\Lambda_\epsilon$ determines $\delta F$ exactly.

Since $F$ is bounded and $\delta F$ is \sP, it
follows from Cor.~\ref{cor-exact} that $F$ is weakly pattern
equivariant, and hence that $s_\epsilon$ is a shape conjugacy. 
\end{proof}

In view of the preceding discussion we can also formulate the last theorem in the following way:
Two FLC Delone sets $\Lambda$ and $\Lambda'$ are pointed topologically conjugate if and only if
for each $\epsilon >0$ there exists a FLC Delone set $\Lambda_\epsilon$ such that
  \begin{enumerate}
  \item $\Lambda$ and $\Lambda_\epsilon$ are mutually locally derivable,
  \item  $\Lambda'$ and $\Lambda_\epsilon$ are mutually asymptotically negligible shape deformations,
  \item within $\epsilon$-distance of each point of $\Lambda'$ lies a point of $\Lambda_\epsilon$ and vice versa.
  \end{enumerate}

\section{Cantor bundles over tori with canonical $\R^d$-action}
\newcommand{\F}{\mathcal F}

It has been known for quite some time that the hull of any FLC tiling is
homeomorphic to a fiber bundle over the $d$-torus whose typical fiber
is a compact totally disconnected space such as a Cantor set \cite{SW}. 
This means
that there is a continuous surjection $X\stackrel{\pi}{\to} \TM^d$  
onto the $d$-torus such that the pre-images $\pi^{-1}(t)$,
$t\in\TM^d$, are all homeomorphic to a single compact totally
disconnected space $\F$, the so-called typical fiber, and that $X$ is
locally a product i.e.\ every point has a neighbourhood $U$ such that
$\pi^{-1}(U)$ is homeomorphic to $\F\times U$. 
However, the
construction of \cite{SW} is based on deforming the tiling to a tiling by
cubes. It yields a  homeomorphism but in general not a topological
conjugacy.  

Let $X\stackrel{\pi}{\to} \TM^d$ 
be any fiber bundle over the $d$-torus. We say that a (continuous) action of
$\R^d$ on the bundle $X$ is {\em canonical} if there exists a regular
lattice $L\subset\R^d$ such that $\pi$ becomes $\R^d$-equivariant
(hence a factor map) when we identify  $\TM^d=\R^d/L$ and equip it
with the action induced by translation on  
$\R^d$ (called a rotation action). 
Note that if the fiber of the bundle is totally disconnected then, 
once we have fixed
the lattice $L$, the canonical action becomes unique.

Any fiber bundle over a $d$-torus with canonical $\R^d$-action has $d$
independent topological eigenvalues. Indeed, the pull-back under $\pi$
of any continuous eigenfunction of  
the rotation action on $\TM^d$ is a continuous eigenfunction of
$(X,\R^d)$. 
In particular the group of eigenvalues of $(\TM^d,\R^d)$ (which is a
lattice in ${\R^d}^*$) is a subgroup of the group of
eigenvalues of $(X,\R^d)$. 

In the above context of a fiber bundle over a $d$-torus with canonical $\R^d$-action
we say that the action is {\em expansive}
if the induced action on a fiber is expansive, that is,
there exists a constant $\epsilon>0$
such that for any $x,y\in \pi^{-1}([0])$, $\sup_{t\in L} d(t\cdot
x,t\cdot y)\leq \epsilon$ implies $x=y$. The largest such constant is called
the expansivity constant. Here $d$ is a metric that induces the
topology. Whereas the expansivity constant will depend on the choice
of $d$, the mere fact that the $\R^d$-action is expansive does not.
There do exist Cantor bundles, such as the dyadic solenoid over the circle, 
with non-expansive canonical dynamics, but these are not homeomorphic to
FLC tiling spaces.

If the $\R^d$ action on $X$ is topologically
transitive and $\pi:X\to\TM^d$ is a factor map, then the action by rotation
on the torus must be 
transitive as well. For dimensional reasons the action by rotation
must then also be locally free. 
Since the action of $\R^d$ on $X$ is then locally free as well, the fibers 
$\pi^{-1}([t])$ must be transversal to it. Furthermore
the action of $t\in\R^d$ provides a homeomorphism between
$\pi^{-1}([s])$ and  $\pi^{-1}([s-t])$ and so all fibers are
homeomorphic. Letting $\Xi$ denote the fiber on $[0]$ we see that for 
small enough open balls $U\in \TM^d$, $\Xi\times U\ni
(\xi,[t])\mapsto t\cdot \xi \in \pi^{-1}(U)$ (we lift the ball $U$ to a
ball in $\R^d$) is a homeomorphism providing a local trivialization
for a bundle structure defined by $\pi$. 
Thus, once we have established that $\pi$ is a factor map onto a
rotation action,  
the only issue will be to show that some pre-image has the topological
characterization we want and that the dynamics is expansive. 
%is homeomorphic to a compact totally disconnected space.
We will do this in the next section.

\section{Proof of Theorems \ref{main} and \ref{main2}}

We prove the following three statements in turn. 
Taken together, they
imply both Theorem~\ref{main} and Theorem~\ref{main2}.

\begin{enumerate}
\item{A repetitive Meyer set in $\R^d$ 
has a relatively dense set of topological eigenvalues.
In particular, it has $d$ linearly independent topological eigenvalues.}
\item{If an FLC Delone set has $d$ linearly independent topological
eigenvalues, then it is topologically conjugate to a bundle
over the $d$ torus with totally disconnected compact fiber and 
expansive canonical $\R^d$ action. 
If the Delone set is repetitive, then the $\R^d$ action is minimal.}
\item{A bundle with totally disconnected compact fiber 
over the $d$ torus with minimal, expansive
    canonical $\R^d$ action is topologically conjugate to the dynamical system of a repetitive
    Meyer set.}\end{enumerate}

\begin{proof}[Proof of Statement 1] 
Let $\Lambda$ be a Meyer set and pick $0< \epsilon < 1$ and $a_0\in
\Lambda$. 
%We know that $\Delta^\epsilon$ is relatively dense. 
Using Lemma~\ref{prop-EV} we will establish that each
  $\beta\in\Delta^\epsilon$ is a continuous eigenvalue by showing that  
  $f(x) := \exp(2 \pi i \beta(x-a_0))$ is \wP. Since  $\Delta^\epsilon$
  is relatively dense this then proves Statement~1.

  Since $|\exp(2 \pi i \beta(a-a_0)) -1 |<\epsilon<1$ for all
  $a\in\Lambda$, the real part of $f(a)$ is positive for all
  $a\in\Lambda$.  There is then a function $\theta: \Lambda \to
  (-\pi/2, \pi/2)$ such that $f(a) = e^{i\theta(a)}$ for all
  $a\in\Lambda$.  Consider the $1$-cochain $\alpha=\delta \theta$ on
  $\T_\Lambda$. It
  satisfies $-\pi < \alpha(c) < \pi$ for each edge $c$. We claim that
  $\alpha = \delta \theta$ is \sP.  Since $\exp(i \alpha)(c) = e^{i
    \theta(b)-i\theta(a)} = \exp(2\pi i \beta(b-a)),$ where $a,b$ are the boundary
  vertices of $c$, we see that $\exp(i \alpha)$ is determined by the
  length and the direction of $c$ and is thus \sP.  But since $-\pi <
  \alpha(c) < \pi$, $\alpha$ is itself \sP.  Hence $\delta \theta$
  satisfies the conditions of Prop.~\ref{prop-2}. Hence for all
  $\eta>0$ there exists $\Lambda_\eta$, locally derivable from
  $\Lambda$, such that $\forall a,b\in\Lambda_\eta$ we have
  $|\theta(a)-\theta(b)|<\eta$. Since $f(a) = e^{i\theta(a)}$ we
  conclude that for all $\eta>0$ there exists $\Lambda_\eta$, locally
  derivable from $\Lambda$, such that $\forall a,b\in\Lambda_\eta$ we
  have $|f(a)-f(b)|<\eta$.

  We claim that $f$ satisfies the conditions of Lemma~\ref{lem-wP},
  which then proves Statement~1. Clearly $f$ is uniformly continuous.
  $\Lambda_\eta$ being locally derivable from $\Lambda$ means that
  there exists a finite set $\{p_1,\cdots,p_k\}$ of patches $p_i$ such that
$\Lambda_\eta = \{x\in\Lambda:\exists i:p_i\subset \Lambda-x\} $.
Let $R$ be large enough so that each $R$-ball
  contains some patch $p_i$. Then, given $x$ there are $i$ and $t\in
  B_{R}(0)$ such that $p_i\subset (\Lambda-x-t)\cap B_R(0)$. In
  particular this means that $x+t\in\Lambda_\eta$.  Hence if
  $(\Lambda-x)\cap B_R(0) = (\Lambda-y)\cap B_R(0)$ then $|f(x)-f(y)|
  = |f(x+t)-f(y+t)|<\eta$ as $x+t$ and $y+t$ belong to
  $\Lambda_\eta$. \end{proof}

Note that we have proved considerably more than just the relative density
of the topological eigenvalues. The topological eigenvalues form a 
group, so the set of topological eigenvalues contains the 
group generated by $\Delta^\epsilon$. In many (but not all) cases,
this group is the same as the set of measurable eigenvalues. In these
cases, all eigenvalues are topological!  This gives a new perspective on 
why the eigenfunctions for substitution tilings and for model sets can 
always be chosen continuous.

\begin{proof}[Proof of statement 2] The $d$ topological eigenvalues
  $\beta_i$ generate a lattice $L^* \subset \R^{d*}$. Let $L \subset
  \R^d$ be the dual (or reciprocal) lattice to $L^*$. The $d$
  eigenfunctions are periodic with period $L$, and their values give a
  map $\pi(\omega) = (\beta_1(\omega),\cdots,\beta_d(\omega))\in
  \R^d/L$ from the hull $\Omega$ to the torus $\TM^d=\R^d/L$. By
  (\ref{eq-EV}) $\pi$ is $\R^d$-equivariant if we consider the
  rotation action on $\TM^d$. The hull is thus a bundle over the
  $d$-torus with canonical $\R^d$-action.

  Let $\Xi = \pi^{-1}([0])$. It is compact as $\Omega$ is compact.  We
  wish to show that it is totally disconnected. For that we only use
  that the hull of an FLC tiling is a matchbox manifold, 
i.e.\ there is a finite open covering $\{\tilde U_i\}_{i\in
    I}$ of $\Omega$ such that $\tilde U_i \cong C_i\times U_i$ where
  $C_i$ are totally disconnected compact sets and $U_i$ open balls of
  $\R^d$ and the $\R^d$-action is given locally by translation in the
  second coordinate: $t\cdot (c,u) = (c,u-t)$ provided $u-t\in
  U_i$. Let $\Xi_i:= \Xi\cap \overline{\tilde U_i}$.  We explained
  above that $\Xi$ is transversal to the action. This means that
  return vectors to $\Xi$ have a length which is bounded from below by
  some strictly positive number $l$. If we choose the sets $U_i$ to be
  of diameter smaller than $l$ then $\Xi_i$ intersects the plaquettes
  $\{c\}\times \overline{U_i}$ at most once. Hence the projection
  $pr_1:C_i\times \overline{U_i}\to C_i$ onto the first factor
  restricts to a continuous bijection between $\Xi_i$ and its image
  $pr_1(\Xi_i)$. Since $\Xi_i$ is compact this bijection is
  bi-continuous.  Since $pr_1(\Xi_i)$ is a compact subset of a totally
  disconnected space it is itself totally disconnected. Thus $\Xi_i$
  is totally disconnected. Since finite unions of totally disconnected
  sets are totally disconnected, $\Xi=\bigcup_i \Xi_i$ is totally
  disconnected.
 
  We next show expansivity. For this we choose the standard tiling metric
  on the hull $\Omega$.  By finite local complexity, if two tilings
  agree exactly at a spot and the nearest neighbor tiles do not agree
  exactly, then the nearest neighbor tiles either differ in label or
  differ in position by at least a fixed quantity $\epsilon$. If two
  tilings $\T_{1,2} \in \Xi$, then either $d(\T_1, \T_2) \ge
  \epsilon/2$ or there is a translate $\T_1'$ of $\T_1$ by less than
  $\epsilon/2$ such that $\T_1'$ and $\T_2$ agree exactly at the
  origin. But then $\T_1'$ and $\T_2$ disagree by at least $\epsilon$
  at the nearest tile where they don't agree exactly, so translates of
  $\T_1'$ and $\T_2$ differ by at least $\epsilon$, so translates of
  $\T_1$ and $\T_2$ disagree by at least $\epsilon/2$. This proves
  that the $\Z^d$ action on $\Xi$ is expansive with expansivity
  constant $\epsilon/2r$, where $r$ is the diameter of a fundamental
  domain of $L$. The last statement is well known.
\end{proof}

It is not difficult to see that if the Delone set is repetitive and does not have $d$ independent
periods,
then the fiber of the bundle has no isolated points and so is a Cantor set. Indeed, by 
minimality $\Xi$ consists either of a single (necessarily finite) orbit, or has no isolated points.
The first case arises precisely if the tiling is totally periodic, with $d$ independent periods. 

We remark that the expansivity of the action could also be concluded
from the work of Benedetti-Gambaudo \cite{BG} who define the concept
of expansivity more generally for tiling dynamical systems on
homogeneous spaces using solenoids.

\begin{proof}[Proof of Statement 3]
Let $\Xi$ be the totally disconnected compact fiber over $[0]\in\TM^d$ (a point we can actually choose) of the bundle $\Omega\to\TM^d$.
The canonical expansive action induces an expansive action of $L\cong\Z^d$ on $\Xi$. 
Let $\epsilon$ be the expansivity radius. Pick a finite 
clopen cover of $\Xi$ by sets of diameter less than $\epsilon$, and label these
sets with a finite alphabet $\aaa = \{1, \ldots, n\}$. To each
element $\xi \in \Xi$ associate an element $u_\xi \in \aaa^{\Z^d}$, where
$u_\xi(t)$ labels which clopen set contains $t\cdot \xi$. The map from $\Xi \to 
\aaa^{\Z^d}$ is injective, so we can view $\Xi$ as a subshift of $\aaa^{\Z^d}$. 
This then makes $\Omega$ topologically conjugate to
a tiling space where each tile is a decorated fundamental domain of the lattice $L$. 
The decoration is given by a letter of the alphabet $\aaa$ but each letter $i$ may as well be
encoded by a finite set $A_i$. Let then $\Lambda$ be the set of vertices of that tiling together with the points of the sets which decorate the tiles. Then $\Lambda$ is MLD with the tiling and furthermore $\Lambda\subset L + \bigcup_{i=1}^n A_i$ showing that it is Meyer.  

Minimality of the bundle implies minimality of the dynamical system of $\Lambda$ and hence its repetitivity.     
\end{proof}

Combining Theorem~\ref{main} with Theorem~\ref{thm-shape} we obtain the following corollary.
\begin{cor}
Let $\Lambda$ be an FLC repetitive Delone set with $d$ independent topological eigenvalues.
For each $\epsilon>0$ there is a Meyer set $\Lambda_\epsilon$ such 
that $\mbox{\rm dist}(\Lambda,\Lambda_\epsilon)\leq \epsilon$, where $\mbox{\rm dist}$ denotes Hausdorff distance between sets of $\R^d$. 
\end{cor}
This does not imply that $\Lambda$ is itself Meyer, as the minimal distance between points of $\Lambda_\epsilon - \Lambda_\epsilon$ depends on $\epsilon$.

\section{Counterexamples}\label{Counter-ex}
In this section we provide examples of point sets with unusual behavior that at first sight seems surprising and was partly suspected to be impossible. Indeed we provide first a counterexample to Lagarias' Problem~4.10 \cite{Lagarias}, namely a repetitive FLC Delone set which is pure point diffractive but not Meyer.  
This example can already be found in \cite{FS}: it is the scrambled Fibonacci tiling with tile lengths $\phi$ and $1$. A variation of this tiling with tile lengths all equal yields an example of a repetitive Meyer set which is pure point diffractive but has non-topological eigenvalues. This is actually a special case of general phenomenon: any letter substitution can be suspended to a (possibly decorated) tiling whose tiles have equal length and so is a Meyer tiling. Then the dynamical spectrum of the $\Z$-action on the shift space coincides with the dynamical spectrum of the $\R$-action on the continuous hull. Bratteli diagram techniques can then be used to determine whether an eigenvalue is topological or not \cite{Durand1,Durand2,Durand3}.
  
For the reader's convenience, we sketch the construction of \cite{FS} in a form that allows variations in tile lengths, and include proofs of some results found in \cite{FS}.  The construction is similar in spirit to an S-adic system, although not technically S-adic itself.

\renewcommand{\k}{\kappa}

\subsection{Scrambled Fibonacci sequences}
Let $\sigma,P_1:\aaa\to \aaa^*$ be two maps on the alphabet $\aaa$.  
%We think of the second map as a supertile map.
Both maps are extended in the usual way as morphisms on $\aaa^*$.  Let
$P_2=P_1\circ\sigma$. Then one obtains $P_2(a)$ upon replacing each
letter $o$ in $\sigma(a)$ by the word $P_1(o)$. We may also think of
$\sigma$ as a rule to compose (or fuse) words of the first generation
$\{P_1(o), o\in\aaa\}$ to form words of the next generation $P_2(a)$.

Now suppose that we have a family of substitutions
$\{\sigma^{(n)}\}_{n\in\N}$ and let $P_0 = \mbox{\rm id}$. The
composition $P_n = P_{n-1}\circ \sigma^{(n)}$ then describes the fusion
of words of generation $n-1$ to words of generation $n$. This is an
example of a fusion rule \cite{FS}.\footnote{Even in one dimension
  this is not the most general type of fusion rule, but it is the one
  we need.} The sequence space associated to the fusion rule is the
subset $X$ of $\aaa^\Z$ of doubly infinite sequences whose factors are
sub-words of some $P_n(o)$, $o\in\aaa$, $n\in\N$. Words of the form
$P_n(o)$ are called $n$th order superletters.
 
If all $\sigma^{(n)}$ are equal to a fixed substitution $\sigma$ then $X$ is the substitution sequence space associated to $\sigma$. 

For concreteness we will consider the Fibonacci substitution $\sigma$
on $\aaa=\{a,b\}$ defined by $\sigma(a) = ab$, $\sigma(b) = a$ and
denote $F_n = P_n$, i.e.\ 
$F_n=\sigma^n$. We denote its sequence space $X$ by $X_F$. Recall that
$F_n(a)$ contains $f_{n+1}$ letters $a$ and $f_{n}$ letters $b$ while
$F_n(b)$ contains $f_{n}$ letters $a$ and $f_{n-1}$ letters $b$. Here
$f_n$ is the $n$th Fibonacci number defined iteratively by $f_0=0$,
$f_1=1$ and
$f_{n+1} = f_n + f_{n-1}$.

Let $N(n)$ be a strictly and fast increasing function on $\N$ and
$\Delta(n) = N(n)-N(n-1)$. Let $A_n = F_{N(n)}$, which we may also
write as $A_n = A_{n-1}\circ\sigma^{\Delta(n)}$, viewing $A_n(o)$ as the
$n$th order superletters to the fusion rule
$\{\sigma^{\Delta(n)}\}_{n\in\N}$.  The sequence space $X_F$
can now also be described as the subset of doubly infinite sequences
in $\aaa^\Z$ whose factors are sub-words of some $A_n(a)$, $a\in\aaa$,
$n\in\N$.

Extend the alphabet with an extra letter $\tilde\aaa = \aaa\cup\{e\}$.
%Letters from the original alphabet $\aaa=\{a,b\}$ will be denoted by
%the variable $o$. 
Let $\k$ denote an arbitrary odd natural number. 
Define the fusion rule by the following family,
$\tilde\sigma^{(n)}:\tilde\aaa \to \tilde\aaa^*$
$$
\tilde\sigma^{(n)}(x) = \left\{\begin{array}{ll}
\sigma^{\Delta(n)}(x) & \mbox{if } x \in \aaa,\: n\:\mbox{odd}  \\
\sigma_{e}^{\Delta(n)}(x) & \mbox{if } x \in \aaa,\: n\:\mbox{even}  \\
\sigma_{r}^{\Delta(n)}(b)  & \mbox{if } x = e
\end{array}\right. 
$$
where $\sigma_{e}^{\Delta(n)}(x)$ is the word obtained from $\sigma^{\Delta(n)}(x)$
by replacing the last letter $b$ with $e$, and
$\sigma_{r}^{\Delta(n)}(b)$ is the word obtained from
$\sigma^{\Delta(n)}(b)$ by rearranging all the $a$'s in front of the
$b$'s, that is $\sigma_{r}^{\Delta(n)}(b)=a^{f_{\Delta(n)}}b^{f_{\Delta(n)-1}}$.
We denote by $S_n(x)$ the $n$th order superletters of this fusion rule. 
That is, for $\k$ odd,
\begin{eqnarray}
\label{eq-S1}
S_\k(o) = S_{\k-1}\circ\sigma^{\Delta(\k)}(o),\mbox{ for }o \in\aaa, & 
S_\k(e) = S_{\k-1}\circ\sigma_{r}^{\Delta(\k)}(b) \\
\label{eq-S2}
S_{\k+1}(o) = S_{\k}\circ\sigma_{e}^{\Delta(\k+1)}(o),\mbox{ for }o \in\aaa, & 
S_{\k+1}(e) = S_{\k}\circ\sigma_{r}^{\Delta(\k+1)}(b)
\end{eqnarray}
The subset $X_S$ of doubly infinite sequences in $\tilde\aaa^\Z$ whose
factors are sub-words of some $S_n(x)$, $n\in\N$
is by definition the space of {\em scrambled Fibonacci
  sequences}. Note that none of the 1-superletters contain the letter $e$,
so this is in fact a subset of $\aaa^\Z$. In fact, none of the odd-order
superletters ever contain an even-order superletter of type $e$, so the
definition of $S_{\k+1}(e)$ is in fact irrelevant. 

In other words, the scrambling consists of two steps. At each odd level we
introduce a germ $S_\k(e)$ made from two periodic pieces. The presence of
$S_\k(e)$ superletters with $\k$ large will serve to eliminate topological
eigenvalues. At even levels we insert a single $S_k(e)$ into $S_{k+1}(b)$.
As long as $\Delta(\k+1)$ grows quickly enough, this makes the $e$-superletters
too rare to affect the measurable dynamics, and in fact tiling spaces
based on $X_S$ will 
prove to be measurable conjugate to those based on $X_F$.

We say that a fusion rule $\{S_n\}$ is recognizable if for each $n$,
each sequence in $X_S$ can be uniquely decomposed into a 
concatenation of sequences of superletters $S_n(x)$. For substitution 
tilings, recognizability is an automatic consequence of non-periodicity.
For fusions, it must be checked separately.

\begin{prop}[\cite{FS}] The scrambled Fibonacci fusion rule is recognizable. 
\end{prop}
\begin{proof}
We prove this by induction. Let $\xi \in X_S$. 
Let $\k$ be odd and assume that $\xi$ can uniquely be written as
a sequence of $(\k-1)$-superletters (this is trivially true for $\k=1$). 
It is easy to see that $\sigma^{\Delta(\k)}(o)$
starts with $a$ and does not contain a $b^2$. Thus we can identify 
the powers $S_{\k-1}(b)^{f_{\Delta(\k)}-1}$ in $\xi$
and regroup $S_{\k-1}(a)^{f_{\Delta(\k)}}S_{\k-1}(b)^{f_{\Delta(\k)}-1}$ to $S_\k(e)$. 
The remaining $S_{\k-1}(a)$'s and $S_{\k-1}(b)$'s can be regrouped by applying
$\Delta(\k)$ times the procedure known from inverting the Fibonacci
substitution, namely first regroup all words $ab$ to $a$ and then
replace the remaining $a$'s (not the new ones!) by $b$. As for the  
Fibonacci substitution one sees that this is the only way to regroup. 

We now show that we can group the sequence into $(\k+1)$-superletters.  
Having grouped the sequence into $\k$-superletters, 
one can identify the $S_{\k+1}(b)$'s by the presence of an $S_\k(e)$. The
remaining ${\k+1}$-supertiles are all of type $a$, since even-order
superletters of type $e$ do not appear. 
\end{proof}

Let $\xi^{(n)}$ denote the $n$th order superletter 
that contains $\xi(0)$. (By recognizability, this is uniquely defined.) 
Let $X_S^*\subset X_S$ be the set of all sequences $\xi$ such that (a) 
the union of the words $\xi^{(n)}$ is all of $\xi$, and not just a half-line,
and (b) only finitely many $\xi^{(n)}$'s are of type $e$. For each $\xi \in
X_S^*$, let $n_0$ be the largest value of $n$ for which $\xi^{(n)}$ is of
type $e$ (or zero if none of these superletters are of type $e$.)

We may define a shift equivariant map
$$ \psi:X_S^* \to X_F$$
in the following way: $\psi(\xi) = \lim_{n\to \infty}\psi_n(\xi)$
where, for $n\geq n_0$, $\psi_n(\xi)$ is obtained from $\xi$ by
replacing $\xi^{(n)}$ with the corresponding Fibonacci
superletter $A_n(o)$, where $\xi^{(n)}$ is of type $o$. 
Note that $\psi_{n+1}(\xi)$ and $\psi_n(\xi)$ agree on the image of 
$\xi^{(n)}$, since neither $\xi^{(n)}$ nor $\xi^{(n+1)}$ is of type $e$. 
The limit is taken in the product topology. Shift-equivariance 
is guaranteed by the fact that the union of the $\xi^{(n)}$'s is all of
$\xi$.  

We cannot expect $\psi$ to be continuous, nor do we have 
a reasonable extension of $\psi$ to all of $X_S$. However, 
$\psi$ is measurable. With appropriate conditions on $N(n)$, 
$X_S^*$ has full measure with respect to the (unique \cite{FS}) invariant
Borel probability measure on $X_S$.
In fact, the frequency of the superletter $S_\k(e)$ in any sequence 
$\xi\in X_S$ is bounded by 
$\phi^{-\Delta(\k+1)}$, since there is only one $S_\k(e)$ per $S_{\k+1}(b)$.
As a result:
\begin{lem}[\cite{FS}] Let $\mu$ be an ergodic shift invariant Borel
  probability measure on $X_S$.  If $\sum_{k\: odd}
  \phi^{-\Delta(k+1)} <\infty$ then $X_S^*$ has full $\mu$-measure.
\end{lem}
\begin{cor}[\cite{FS}] Under the assumptions of the last lemma,
  $\psi$ defines a conjugacy between the shift measure dynamical
  systems $(X_S,\mu,\Z)$ and $(X_F,\Z)$ equipped with the unique
  ergodic Borel probability measure. In particular, $(X_S,\Z)$ is
  uniquely ergodic.
\end{cor}
We recall also from \cite{FS} that the shift dynamical system
$(X_S,\Z)$ is minimal.

\subsection{Scrambled Fibonacci tilings}
We now make tilings out of sequences by suspending the letters to
intervals which we call tiles. 
We give the tile corresponding to $a$ length $L$ and
the tile corresponding to $b$ length $S$. Images of superletters are
called supertiles. 
We will denote the corresponding tiling spaces by $\Omega_F^{L,S}$ and
$\Omega_S^{L,S}$. As for the discrete systems on obtains that if
$\sum_{\k\: odd} \phi^{-\Delta(\k+1)} <\infty$ then $(\Omega_S^{L,S},\R)$
is uniquely ergodic and measure conjugate to $(\Omega_F^{L,S},\R)$ so
that in particular their  dynamical spectra coincide. Furthermore, if
$L\phi + S = L'\phi + S'$ then deformation theory shows that
$(\Omega_F^{L,S},\R)$ is topologically conjugate to
$(\Omega_F^{L',S'},\R)$ \cite{RS}. 
We are asking which eigenvalues of $(\Omega_S^{L,S},\R)$ are
topological and the answer will depend on the values of $L,S$.  
To answer this question we make use of the following corollary to 
Lemma~\ref{lem-wP}.
By a return vector between $n$-supertiles we mean a vector $v$ such
that in some high order supertile we find the $n$-supertile $S_n(o)$
at position $x$ and at position $x+v$. 
For a real number $r$ let  $\|r\|$ denotes the distance from $r$ to
the nearest integer. 

\begin{cor}\label{cor-wP}
$\beta$ is a topological eigenvalue if and only if for all 
$\epsilon>0$ there exists an $n_0$ such that for all return vectors 
$v$ between supertiles of order $n\geq n_0$ one has 
$\| \beta(v)\|\leq \epsilon$. 
\end{cor}

This corollary applies to the Fibonacci tiling and to the scrambled
Fibonacci tilings for all values of $L,S$. For the Fibonacci tiling with
$L=\phi$ and $S=1$ it is known that all eigenvalues are
topological and then one can read also from the corollary that a
necessary and sufficient condition for $\beta\in\R^*=\R$ to be an eigenvalue is
$\|\beta \phi^n\|\to 0$ which is equivalent to
$\beta\in\frac1{\sqrt{5}}\Z[\phi]$. After rescaling the tiles by $\sqrt{5}$
and doing a shape change that preserves $\phi L + S$, we get that the
group of eigenvalues for $(\Omega_F^{1,1},\Z)$ is
 $\Z[\phi]$.  
\begin{thm} [\cite{FS}] 
Suppose that $\Delta(\k)\geq N(\k-1)$ for odd $\k$  and that $\sum_{\k\: odd}\phi^{-\Delta(\k+1)}<\infty$.
If $L=\phi$ and $S=1$ then the only topological eigenvalue is 0. 
\end{thm}
\begin{proof}
Suppose that $\beta\in\R^*=\R$ is a topological eigenvalue.  By
Corollary~\ref{cor-wP},  the maximum value of 
$\|\beta v\|$, where $v$ is a return vector to 
$\k-1$-supertiles of type $a$, must go to zero as $\k \to \infty$. 
Likewise, $\|5 \beta v\| \le 5 \|\beta v\|$ must go to zero. Note that the scrambling was made
in such a way that $S_n(a)$ has the same length as $F_{N(n)}(a)$, namely
$\phi^{N(n)+1}$.
Since $S_{\k}(e)$ contains $f_{\Delta(\k)}$ consecutive $S_{\k-1}(a)$'s, and 
since $\Delta(\k) \ge N(\k-1)$, 
%S_{\k-1}(a)^{f_{\Delta(\k)}}S_{\k-1}(b)^{f_{\Delta(\k)-1}}$ we see that 
$v$ can take on the particular values 
$v_1=f_{N(\k-1)-1}|S_{\k-1}(a)|= f_{N(\k-1)-1}\phi^{N(\k-1)+1}$ and 
$v_2=f_{N(\k-1)-2}|S_{\k-1}(a)| = f_{N(\k-1)-2}\phi^{N(\k-1)+1}$.
Using the identity $f_{n} = (\phi^n - (-\phi)^{-n})/\sqrt{5}=(\phi^{n}-(-\phi)^{-n})(\phi+\phi^{-1})/5$, we obtain
$$ \|5\beta v_m\| = 
\|\beta (\phi^2+1)(\phi^{2N(\k-1)-m} - (-1)^{N(\k-1)-m} \phi^{m})\|,$$
where $m=1$ or 2. 
Since $\beta$ is a measurable eigenvalue,
$\|\beta \phi^n\| \to 0$ as $n \to \infty$. This means that $\|\beta v_m \|$
can only go to zero if  
$\beta (\phi^2+1) \phi^m\in\Z$. However, $\phi$ is irrational, so this
cannot be true for both $m=1$ and $m=2$ unless $\beta=0$.\end{proof}
\begin{cor}
Suppose that $\Delta(\k)\geq N(\k-1)$ for odd $\k$ and that $\sum_{\k\: odd}\phi^{-\Delta(\k+1)}<\infty$.
Then the scrambled Fibonacci tiling with lengths $|a|=\phi$ and 
$|b|=1$ is repetitive and pure point diffractive but not Meyer.  
\end{cor}
\begin{thm} 
Suppose that $\Delta(\k)\geq N(\k-1)$ for $\k$ odd  and that $\sum_{\k\: odd}\phi^{-\Delta(\k+1)}<\infty$.
If $L=S=1$ then the group of topological eigenvalues is $\Z$. 
\end{thm}
\begin{proof} Since $L=S=1$ all return vectors are integers and hence any integer $\beta$ is a topological eigenvalue. To prove that the rest of the dynamical spectrum is not topological we use 
the same idea as for the preceding theorem. The difference is that now 
$$|S_{\k-1}(a)| = f_{N(\k-1)+1} + f_{N(\k-1)}= 
\frac1{\sqrt{5}}(\phi^{N(\k-1)+2}-(-\phi)^{-N(\k-1)-2}).
$$
Accordingly we obtain as necessary criterion for $\beta$ to be a topological eigenvalue 
that
$$\|5\beta v_m\|= \|\beta (\phi^{N(\k-1)-m}-(-\phi)^{-N(\k-1)+m})(\phi^{N(\k-1)+2}-(-\phi)^{-N(\k-1)-2})\|  \stackrel{\k \to +\infty}{\longrightarrow} 0$$
 for $m = 1$ or 2.
 Since  $\|\beta\phi^n\|\to 0$ as $n\to  \pm\infty$, we conclude that 
 $\beta (\phi^{-m-2}+(-1)^{m}\phi^{m+2})\in\Z$ for $m= 1$ or $m=2$. Since $\beta(\phi^{-3}-\phi^3) = -4 \beta$ and $\beta(\phi^{-4}+\phi^4)=7\beta$ are integers, $\beta \in \Z$. 
\end{proof}
\begin{cor}
Suppose that $\Delta(\k)\geq N(\k-1)$ for odd $\k$ and that
$\sum_{\k\: odd}\phi^{-\Delta(\k+1)}<\infty$. Then the scrambled
Fibonacci tiling with equal tile lengths is a repetitive pure point
diffractive Meyer tiling having some eigenvalues that are not
topological. The topological eigenvalues form a direct summand
subgroup of rank $1$ in the full group of eigenvalues which has rank
$2$.   
\end{cor}

\subsection{A non-Meyer shape change of a Meyer tiling}

\begin{thm}
  There are FLC sets that are topologically conjugate to Meyer sets but
  are not themselves Meyer.
\end{thm}

\begin{proof} Consider the irreducible substitution $\sigma$ on three
letters, with $\sigma(a) = abca$, 
$\sigma(b) = abb$, and $\sigma(c)= ac$.  Since each substituted letter begins
with $a$, the first cohomology of the resulting tiling space is the direct 
limit of the transpose of the 
subsitution matrix $M=\begin{pmatrix}2&1&1 \cr 1&2&0 \cr 1&0&1
\end{pmatrix}$ \cite{BD}. 
The eigenvalues
of this matrix are $\lambda_1 \approx 3.247$, $\lambda_2 \approx 1.555$, 
and $\lambda_3 \approx .1981$.  Call the eigenvectors $\xi_1$, $\xi_2$
and $\xi_3$. 
Since the third eigenvalue is less than 1 in magnitude, the corresponding
eigenvector is represented by a weakly exact cochain that describes a shape 
conjugacy \cite{CS2}.
The first and second eigenvalues
are greater than 1, and no combination of the two can be represented by 
a weakly exact cochain. 

Pick a bi-infinite sequence coming from the substitution, and consider two
tilings corresponding to that sequence. In one, all of the tiles have 
length 1. In the other, the tile lengths are $\begin{pmatrix}1 \cr 1 \cr 1
\end{pmatrix}$
plus a small multiple of the third eigenvector, chosen so that
the resulting tile lengths are not rationally related. 
The left endpoints of the
$a$ tiles give us point sets $\Lambda$ and $\Lambda'$. 

Since $\Lambda$ is a subset of the integer lattice, it is Meyer. 
$\Lambda'$ is topologically conjugate to $\Lambda$. We will show that
$\Lambda'$ is not Meyer. 

For each finite word $w$ in our space of sequences, let $\ell(w)$ be
the ``population vector'' in $\Z^3$ that counts the number of $a$'s, 
$b$'s and $c$'s in $w$. Consider the possible population vectors of words
of length $n$. The deviation from the average population for a word of length
$n$ is governed by eigenvalues $\lambda_2$ and $\lambda_3$.  Since $|\lambda_3|<1$, 
there is an upper bound to the inner product of $\ell(w)$ with
$\xi_3$. However, since $\lambda_2>1$, the range of possible inner products with $\xi_2$ increases 
with $n$, going as $n^{\ln(\lambda_2)/\ln(\lambda_1)}$. 
Thus the number of possible population vectors for a given length 
$n$ is bounded above and below by constants times $n^{\ln(\lambda_2)/\ln(\lambda_1)}$. 
\footnote{If you apply the substitution $k$ times, the length of a word grows as $\lambda_1^k$, while
the deviation of the population from the average population grows as $\lambda_2^k$.
For a rigorous proof of the $n^{\ln(\lambda_2)/\ln(\lambda_1)}$ law for substitutions on two letters,
see \cite{KSS}. The same proof works for any substitution where there are exactly two eigenvalues bigger
than 1.}
Since the lengths of
the different tiles are irrationally related, each population vector gives
a different spacing (in the tiling) between letters that are $n$ apart (in the
sequence). This means that the number of spacings between letters that are 
at most $n$ apart grows faster than $n$. By the pigeonhole principle, the 
set of possible spacings of letters cannot be uniformly discrete, so the
set of all tile vertices is not Meyer. 

The same argument applies to substituted letters, implying that the left
endpoints of the 1-supertiles are not a Meyer set. However, every 1-supertile
begins with an ``a'' tile, so the left endpoints of the 1-supertiles
are a subset of $\Lambda'$. Since a relatively dense subset of a Meyer set is Meyer,
$\Lambda'$ cannot be Meyer. 
\end{proof}

\end{document}